\pdfoutput=1
\documentclass{article}

\usepackage{siunitx}

\usepackage{amsmath}
\usepackage{amssymb}
\usepackage{amsthm}
\usepackage{amsfonts}
\usepackage{graphicx}
\usepackage[sort, numbers]{natbib}
\usepackage{xcolor}
\usepackage{multirow}

\newtheorem{lemma}{Lemma}

\newtheorem{assumption}{Assumption}

\newcommand{\eigenvalue}{s}

\title{Robust signal dimension estimation via SURE}

\usepackage{authblk}

\author[1]{Joni Virta}
\author[1]{Niko Lietz\'en}
\author[1]{Henri Nyberg}
\affil[1]{Department of Mathematics and Statistics, University of Turku, Finland}

\date{}

\begin{document}

\maketitle

\begin{abstract}
The estimation of signal dimension under heavy-tailed latent factor models is studied. As a primary contribution, robust extensions of an earlier estimator based on Gaussian Stein's unbiased risk estimation are proposed. These novel extensions are based on the framework of elliptical distributions and robust scatter matrices. Extensive simulation studies are conducted in order to compare the novel methods with several well-known competitors in both estimation accuracy and computational speed. The novel methods are applied to a financial asset return data set. 
\end{abstract}

\section{Introduction}\label{sec:introduction}

\subsection{Premise}

Modern data sets exhibit increasingly large sizes and often the first step in analysing them is the implementation of suitable dimension reduction procedures, for example, principal component analysis (PCA). A fundamental question pertaining to any such reduction is the choosing of the correct amount of latent components to retain: underestimating their amount leads to losing important information, whereas picking too many components inevitably leads in the later stages of the analysis to modelling of noise and increased computational burden that could have been avoided with a more careful choice of the dimension.

Besides being large in size and volume, in many applications, such as economics and finance, it is common for the data sets to display heavier tails than possessed by the Gaussian distribution. Consider, for example, stock market returns which are often modelled with distributions having infinite variance \cite{borak2011models}. In such cases, the estimation of the latent dimension in dimension reduction is further complicated as one cannot rely any more on the covariance matrix, on whose eigenvalues many of the standard dimension estimation methods rely (a review is given later in this section). 

With the above scenario in our mind, the objective of the current paper is to study and develop dimension estimation in the context of multivariate data sets exhibiting arbitrarily heavy tails. While our treatment is not high-dimensional in the usual sense, we still put a special emphasis on obtaining a method that is computationally scalable also to data with the number of variables in hundreds rather than in tens as such dimensionalities are very common in the applications listed earlier. We will base our theoretical framework on the concepts of ellipitical family and Stein's unbiased risk estimation, elaborated in more detail next.

\subsection{Elliptical latent factor model}

We assume that our data $x_1, \ldots , x_n$ is an i.i.d. sample of $p$-variate vectors generated from the elliptical latent factor model
\begin{align}\label{eq:elliptical_model}
    x_i = \mu + V D z_i,
\end{align}
where $\mu \in \mathbb{R}^p$, $V \in \mathbb{R}^{p \times p}$ is an orthogonal matrix, $z_i$ obeys a spherical distribution \citep{fang2018symmetric}, i.e., $z_i \sim O z_i$ for any orthogonal matrix $O \in \mathbb{R}^{p \times p}$, and $D \in \mathbb{R}^{p \times p}$ is a diagonal matrix with the diagonal elements $\sigma_1 \geq \cdots \geq \sigma_d > \sigma = \cdots = \sigma$ with $\sigma > 0$. Conceptually, the model says that the observed $x_i$ are obtained by mixing the principal components $D z_i$ with the matrix $V$ and by applying a location shift $\mu$. The final $p - d$ principal components in $D z_i$ are orthogonally invariant, meaning that they are essentially ``structureless'' and, as is typical, we view them as \textit{noise}. Thus the main objective in the model \eqref{eq:elliptical_model} is to estimate the latent \textit{signals}, i.e., the first $d$ elements of $D z_i$ along with their number $d$.

The model \eqref{eq:elliptical_model} gets still more intuitive form in the special case where $z_i$ obeys the standard Gaussian distribution, the most well-known example of an elliptical distribution. In this case model \eqref{eq:elliptical_model} reduces to
\begin{align}\label{eq:basic_model}
    x_i = \mu + V_0 y_i + \varepsilon_i,
\end{align}
where the loading matrix $V_0 \in \mathbb{R}^{p \times d}$ contains the $d$ first columns of $V$ as its columns,
\begin{align*}
    y_i \sim \mathcal{N}_d \{ 0, \mathrm{diag}(\sigma_1^2 - \sigma^2, \ldots , \sigma_d^2 - \sigma^2 ) \},
\end{align*}
and $\varepsilon_i \sim \mathcal{N}_p(0 ,\sigma^2 I_p) $. Model \eqref{eq:basic_model} reveals that in the Gaussian case the $d$-dimensional signal residing in the column space of $V_0$ is explicitly corrupted with the noise vectors $\varepsilon_i$. Note that this additive representation does not apply to any other elliptical distribution. The standard method of extracting the factors $z_i$ (or the corresponding subspace) in the Gaussian model \eqref{eq:basic_model} is through PCA. Namely, one computes the first $d$ eigenvectors of the covariance matrix of $x_i$ and projects the observations onto their span. However, the success of this procedure hinges crucially on the knowledge of the latent dimension $d < p$, usually unknown to us in practice. As the misspecification of the dimension has ill consequences in practice (either missing part of the signal or riddling our estimates of the latent factors with noise), an important part of solving the factor model $\eqref{eq:basic_model}$ is the accurate estimation of $d$. We next review a particular method for accomplishing this under the model \eqref{eq:basic_model}, on which our subsequent developments are also based.

\subsection{Stein's unbiased risk estimate}

In \citep{ulfarsson2015selecting}, the latent dimension $d$ in Gaussian PCA was estimated through the application of the \textit{Stein's unbiased risk estimate} (SURE), a general technique for determining the optimal values of tuning parameters (of which the latent dimension $d$ is an example) of estimation procedures.

Ignoring the model \eqref{eq:basic_model} for a moment, we briefly recall the basic idea behind the SURE: Assume that we observe the i.i.d. univariate $w_1, \ldots , w_n$, generated as $w_i = a_i + e_i$, where $a_i \in \mathbb{R}$ are constant and the errors satisfy $e_i \sim \mathcal{N}(0, \tau^2)$ for some $\tau^2 > 0$. Let further to each $a_i$ correspond an estimator $\hat{a}_i(w)$, viewed as a (differentiable) function of the data $w = (w_1, \ldots, w_n)$. Then, the SURE $R$ corresponding to the estimators $a_i$ is defined as
\begin{align}\label{eq:sure_basic}
    R = \frac{1}{n} \sum_{i = 1}^n \{ w_i - \hat{a}_i(w) \}^2 + \frac{2 \tau^2}{n} \sum_{i = 1}^n \frac{\partial}{\partial w_i} \hat{a}_i(w) - \tau^2.
\end{align}
In their celebrated paper \citep{stein1981estimation}, Stein proved that $\mathrm{E}(R) = (1/n) \sum_{i=1}^n \mathrm{E} \{ \hat{a}_i(w) - a_i \}^2$, showing that $R$ is an unbiased estimator of the risk associated with the estimators $\hat{a}_i$, in the complete absence of the true means $a_i$. A multivariate version of SURE was in \citep{ulfarsson2015selecting} adapted to the PCA model \eqref{eq:basic_model} (conditional on the $y_i$) to estimate the expected risk associated with any particular choice of the latent dimension $d$. The estimate of $d$ is then chosen to be the dimensionality for which the risk is minimized. See also the earlier work \cite{ulfarsson2008dimension}.

The obtained estimator was in \citep{ulfarsson2015selecting} shown to be highly successful under the Gaussian model \eqref{eq:basic_model}. However, it is clear that the estimator cannot obtain the same level of efficiency under the wider elliptical model \eqref{eq:elliptical_model}. There are two reasons for this: (i) the SURE-criterion was in \citep{ulfarsson2015selecting} derived strictly under the Gaussian assumption and, more importantly, (ii) many standard elliptical distributions (e.g., multivariate Cauchy distribution) do not have enough finite moments so that the covariance matrix, on which the SURE-criterion is based, would even exist. Hence, the estimator of \citep{ulfarsson2015selecting} is not even necessarily well-defined under the elliptical model \eqref{eq:elliptical_model} (on the population level)!

\subsection{The scope of the current work}\label{subsec:scope}

The primary objective of the current work is to provide a workaround for the previous issue by deriving a robust version of the SURE-criterion that allows for effective dimension estimation under the elliptical model \eqref{eq:elliptical_model} in the presence of heavy-tailed distributions. As described earlier, such procedures are highly called for in applications such as finance, where assumptions on finite moments are usually deemed unreasonable. Our robust extension of the SURE-criterion is carried out via a plug-in strategy where the covariance matrix in the Gaussian SURE-criterion is replaced with a suitable \textit{scatter matrix}. Especially popular in the community of robust statistics, scatter matrices are a class of statistical functionals that measure the dispersion/scatter/variation in multivariate data while (usually) being far more robust to the impact of heavy tails and outliers than the covariance matrix, see Section \ref{sec:example_functionals} for their definition and several examples. We consider three different plug-in estimators, depending in which form of the Gaussian SURE-criterion the scatter matrix is plugged in. The first two options lead to analytically simple estimators that depend on the data only through the eigenvalues of the used scatter matrix, much like the classical estimators of dimension (see below). Whereas, the third strategy is more elaborate and involves computing particular derivatives of the scatter functional (and the companion location functional).

As our secondary objective, we conduct an extensive simulation study where the proposed methods are compared to each other and to several (families of) competing estimators from the literature. These competitors include: (i) The classical estimator based on successive asymptotic hypothesis testing for the equality of the final eigenvalues of a chosen scatter matrix (testing for \textit{subshpericity}), see \citep{schott2006high,nordhausen2021asymptotic}. (ii) Variation of the previous estimator where the null distributions are bootstrapped instead of relying on asymptotic approximations \citep{nordhausen2021asymptotic}. (iii) The general-purpose procedure for inferring the rank of a matrix from its sample estimate known as the \textit{ladle}, which we apply to select scatter matrices, see \citep{luo2016combining}. (iv) The SURE-estimator of \citep{ulfarsson2015selecting} which can be seen as the non-robust version of our proposed estimator. We are not aware of comparisons of similar magnitude being conducted earlier in the literature. 

To summarize the results of our simulation study (given in Section \ref{sec:simulations}), they reveal that the SURE-based robust methodology for the determination of the latent dimension is: (i) \textit{Accurate}, achieving comparable or superior estimation results to its competitors in various data scenarios. (ii) \textit{Flexible}, that is, it allows the free selection of the used robust scatter matrix. This is in strict contrast to its closest competitor, the asymptotic hypothesis test mentioned above, which is (for theoretical reasons) ``locked'' to operate with a specific, slow-to-compute scatter matrix. (3.) \textit{Fast}, requiring no bootstrap replicates or any kind of resampling. Due to these three properties, we find the method an especially attractive tool for data-rich large-scale applications.


\subsection{Organization of the manuscript}

The manuscript is organized as follows. In Section \ref{sec:sure_pca} we recall the Gaussian SURE-criterion of \cite{ulfarsson2015selecting} for estimating the latent dimension. In Sections \ref{sec:example_functionals} and \ref{sec:derivative_extension} we propose three different robust extensions of the criterion through the use of different pairs of location and scatter functionals. Sections \ref{sec:simulations} and \ref{sec:real_data} contain the simulation study and an empirical (financial) example on asset returns, respectively. In Section \ref{sec:discussion} we conclude with some future research ideas. The proofs of all technical results are collected in Appendix~\ref{sec:proofs}.

\section{SURE criterion for Gaussian PCA}\label{sec:sure_pca}

In this section, we recall how the SURE-criterion can be used to estimate the latent dimension $d$ under the Gaussian model \eqref{eq:basic_model}. Our derivation of the criterion differs from the original version \citep{ulfarsson2015selecting} in that we employ empirical centering of the data, whereas \citep{ulfarsson2015selecting} did not. We made this change to the method as it is unreasonable to assume that the true location of the data is known in practice. As a consequence, our formula for Gaussian SURE in Lemma \ref{lem:sure_pca} differs non-trivially from the one given in \citep{ulfarsson2015selecting}.

Due to the empirical centering, we assume, without loss of generality, that $\mu = 0$ throughout the following. As in \citep{ulfarsson2015selecting}, we use Stein's Lemma to construct an unbiased estimator of the risk associated with estimating the signals $V_0 y_i$ by their reconstructions $\hat{x}_i$ based on the first $k$ principal components. Assuming that $k = 1, \ldots , p$ is fixed from now on and letting $U_k \in \mathbb{R}^{p \times k}$ denote a matrix of (any) first $k$ orthogonal eigenvectors of the covariance matrix $S_0 := (1/n) \sum_{i=1}^n (x_i - \bar{x})(x_i - \bar{x})'$, the reconstructions can be written as $\hat{x}_i \equiv \hat{x}_{i}(k) = t_0 + P_k (x_i - t_0)$ where $P_k := U_k U_k'$ is the orthogonal projection onto the space spanned by the first $k$ eigenvectors of $S_0$ and $t_0 := (1/n) \sum_{i=1}^n x_i$ is the mean vector. For convenience, we replicate an intermediate result towards the final Gaussian SURE-criterion below as Lemma~\ref{lem:sure_pca}. In the lemma the reconstructions $\hat{x}_i$ are treated as functions of the original data $x_1, \ldots , x_n$ and the result implicitly assumes the former to be differentiable in the latter, sufficient conditions for which will be discussed later in Section~\ref{sec:example_functionals}. In Lemma \ref{lem:sure_pca}, and throughout the paper, $\| \cdot \|$ denotes the Euclidean norm.

\begin{lemma}\label{lem:sure_pca}
Under model \eqref{eq:basic_model}, the quantity
\begin{align*}
    R_{1k} := \mathrm{tr} \{ ( I_p - P_k ) S_0 \} + \frac{2 \sigma^2}{n} \sum_{i = 1}^n \sum_{j = 1}^p  \frac{\partial}{\partial x_{ij}} \hat{x}_{ij} - p \sigma^2
\end{align*}
is an unbiased estimator of the risk $(1/n) \sum_{i = 1}^n \mathrm{E} \| \hat{x}_i - V_0 y_i \|^2$.
\end{lemma}

The two $k$-dependent terms of $R_{1k}$ in Lemma \ref{lem:sure_pca} have natural interpretations: The term $\mathrm{tr} \{ ( I_p - P_k ) S_0 \}$ measures the total variation of the data in directions orthogonal to the first $k$ eigenvectors and takes large values when the used number of eigenvectors is insufficient to capture the full $d$-variate latent signal. The quantity $(1/n) \sum_{i = 1}^n \sum_{j = 1}^p  ( \partial / \partial x_{ij} ) \hat{x}_{ij} $ measures the average influence an observation has on their own reconstruction and is often interpreted as the generalized degrees of freedom of the model, where large values indicate overfitting to the data set, see \citep{tibshirani2012degrees} (in the extreme case with $d = p$ we actually have $\hat{x}_{ij} = x_{ij}$). Thus, $R_{1k}$ can be seen to be similar in form to Akaike's information criterion (AIC) (and other related information criteria), whose two terms also measure model fit and model complexity, respectively.

To apply the criterion $R_{1k}$ in practice, we require an expression for the partial derivatives in Lemma \ref{lem:sure_pca}. As is shown later in the context of Lemma \ref{lem:reconstruction_derivative} in Section \ref{sec:derivative_extension}, the partial derivatives exist under the assumption that the eigenvalues of $S_0$ are simple (which holds almost surely under the model \eqref{eq:basic_model}), and have the forms presented next. See \cite{ulfarsson2008dimension, ulfarsson2015selecting} for similar results.

\begin{lemma}
Under model \eqref{eq:basic_model}, the quantity
\begin{align*}
    R_{2k} := \mathrm{tr} \{ ( I_p - P_k ) S_0 \} + \frac{2 \sigma^2}{n} \sum_{j = 1}^k \sum_{\ell = k + 1}^p \frac{\eigenvalue_j + \eigenvalue_\ell}{\eigenvalue_j - \eigenvalue_\ell}
    + \frac{\sigma^2}{n} \{ 2 p + 2 ( n - 1 ) k - n p \} ,
\end{align*}
where $\eigenvalue_1 > \cdots > \eigenvalue_p$ are the eigenvalues of $S_0$, is an unbiased estimator of the risk $(1/n) \sum_{i = 1}^n \mathrm{E} \| \hat{x}_i - V_0 y_i \|^2$.
\end{lemma}

To apply the criterion $R_{2k}$ in practice, an estimator for the unknown error variance $\sigma^2$ is needed and several feasible alternatives exist. For example, \citep{luo2021order} used, in a similar context, the median of the final $\lfloor p/2 \rfloor$ eigenvalues of $S_0$. The resulting estimator is accurate but makes the implicit assumption that $d \geq \lceil p/2 \rceil$. To avoid such difficult-to-verify conditions, we prefer to instead use the final eigenvalue $\eigenvalue_p$ of $S_0$ as the estimator of the noise variance, imposing minimal assumptions on the latent dimensionality (i.e., that $d < p$). Naturally, the price to pay is that $\eigenvalue_p$ suffers from underestimation in finite samples. Note that, to combat the underestimation, \citep{ulfarsson2008dimension} proposed an alternative estimator of $\sigma^2$ based on the limiting spectral distribution of the covariance matrix under high-dimensional Gaussian data. Mimicking this strategy is not viable in our scenario as any results on the limiting spectral distributions of the scatter matrices used in Section \ref{sec:example_functionals} are still scarce in the literature.

Plugging in the estimator $\eigenvalue_p$ and observing that $\mathrm{tr} \{ ( I_p - P_k ) S_0 \} = \sum_{\ell = k + 1}^p \eigenvalue_\ell$ now leads to two different sample forms for the SURE criterion for Gaussian PCA:

\begin{align}\label{eq:two_forms_for_sure}
\begin{split}
    \hat{R}_{1k} :=& \sum_{\ell = k + 1}^p \eigenvalue_\ell + \frac{2 \eigenvalue_p}{n} \sum_{i = 1}^n \sum_{j = 1}^p  \frac{\partial}{\partial x_{ij}} \hat{x}_{ij} - p \eigenvalue_p,\\
    \hat{R}_{2k} :=& \sum_{\ell = k + 1}^p \eigenvalue_\ell + \frac{2 \eigenvalue_p}{n} \sum_{j = 1}^k \sum_{\ell = k + 1}^p \frac{\eigenvalue_j + \eigenvalue_\ell}{\eigenvalue_j - \eigenvalue_\ell}
    + \frac{\eigenvalue_p}{n} \{ 2 p + 2 ( n - 1 ) k - n p \}.
\end{split}
\end{align}
The ``hat'' notation for $\hat{R}_{1k}, \hat{R}_{2k}$ signifies the fact that they have been obtained from $R_{1k}, R_{2k}$ by replacing the unknown $\sigma^2$ with its estimator $\eigenvalue_p$. In the following sections we obtain outlier-resistant alternatives to both $\hat{R}_{1k}$ and $\hat{R}_{2k}$ via plugging in robust measures of location and scatter in place of the mean vector and covariance matrix in \eqref{eq:two_forms_for_sure}. In addition, we will also consider an ``asymptotic'' version of the criterion $\hat{R}_{2k}$,
\begin{align}\label{eq:asymptotic_sure}
    \hat{R}_{3k} := \sum_{\ell = k + 1}^p \eigenvalue_\ell + \eigenvalue_p (2 k - p),
\end{align}
where the terms of the order $o_p(1)$ (in the asymptotic regime where $n \rightarrow \infty$) have been removed. Note that even though we might have $\eigenvalue_j - \eigenvalue_\ell \rightarrow_p 0$ for some indices $j, \ell$, the limiting distribution of $\sqrt{n}(\eigenvalue_j - \eigenvalue_\ell) $ for such indices is absolutely continuous (with respect to the Lebesgue measure) \citep{anderson1963asymptotic}, meaning that the impact of the double sum in $\hat{R}_{2k}$ can be expected to be negligible for large $n$.

\section{Robust plug-in SURE criteria}\label{sec:example_functionals}

Plug-in-techniques are a typical way to create outlier-resistant versions of standard multivariate methods in the community of robust statistics, see, for example, \citep{croux2000principal,nordhausen2015cautionary}. In this spirit, we replace the mean vector $t_0$ and the covariance matrix $S_0$ in the SURE criteria \eqref{eq:two_forms_for_sure}, \eqref{eq:asymptotic_sure} with a pair $(t, S)$ of \textit{location and scatter functionals} \citep{oja2010multivariate}, the definitions of which we recall next. Letting $F$ be an arbitrary $p$-variate distribution, a location functional (location vector) $t$ is a map $F \mapsto t(F) \in \mathbb{R}^p$ such that, for any invertible $A \in \mathbb{R}^{p \times p}$ and $b \in \mathbb{R}^p$, we have $t(F_{A, b}) = A t(F) + b $ where $F_{A, b}$ is the distribution of the random vector $A x + b$ and $x \sim F$. Similarly, a scatter functional (scatter matrix) $S$ is a map taking values in the space of positive semi-definite matrices and obeying, for any invertible $A \in \mathbb{R}^{p \times p}$ and $b \in \mathbb{R}^p$, the transformation rule $S(F_{A, b}) = A S(F) A'$. These transformation properties are typically referred to as \textit{affine equivariance}.

Location and scatter functionals mimic the properties of the mean vector and the covariance matrix and typically measure some aspects of the center and spread of a distribution, respectively. In particular, if $F$ is the elliptical model \eqref{eq:elliptical_model}, then $t(F) = \mu$ and $S(F) = \tau_{S, F} V D^2 V'$ for all location and scatter functionals $(t, S)$ for which $t(F)$ and $S(F)$ exist, where the scalar $\tau_{S, F} > 0$ depends on both the exact distribution of the spherical $z_i$ and on the used scatter functional, see \citep[Theorem 3.1]{oja2010multivariate}. Hence, all choices of $(t, S)$ estimate, up to scale, the same quantities under the elliptical model, implying that the replacing of the mean vector and the covariance matrix in SURE with the pair $(t, S)$ is warranted (at least in the Gaussian special case \eqref{eq:basic_model} of the elliptical model, under which the SURE criteria in Section \ref{sec:sure_pca} were derived). Note that this equivalence of different $(t, S)$-pairs does not necessarily mean that the dimension estimates given by different choices of $(t, S)$ should always be equal. Namely, the equivalence indeed holds under the population level model \eqref{eq:elliptical_model}, but in practical situations the accuracy of the estimates is greatly influenced by the finite-sample properties (in particular, robustness properties) of the used location and scatter functionals. This is clearly evident in the simulation results of Section \ref{sec:simulations}. 

Examples of popular location and scatter functionals are given later in this section and we assume, for now, that we have selected some robust location-scatter pair $(t, S)$. Outlier-resistant versions of the forms $\hat{R}_{2k}$ and $\hat{R}_{3k}$ of the Gaussian SURE criterion in \eqref{eq:two_forms_for_sure} and \eqref{eq:asymptotic_sure} are then straightforwardly obtained. Namely, we simply replace the eigenvalues $\eigenvalue_j$ of the covariance matrix $S_0$ with the eigenvalues of the scatter functional $S$ in the definitions. Note that while the location functional $t$ does not play an explicit role in this construction, it is usually a part of the definition of $S$, see for example the spatial median and the spatial sign covariance matrix later in this section.

As an alternative to the above, rather simplistic plug-in estimators, we consider also a more elaborate extension based on the form $\hat{R}_{1k}$ of SURE in \eqref{eq:two_forms_for_sure} where, in addition to $S_0$, we also replace the partial derivatives $( \partial / \partial x_{ij} ) \hat{x}_{ij}$ with their counterparts based on the robust pair $(t, S)$. That is, the robust version of $\hat{R}_{1k}$ uses the reconstruction estimates $ \hat{x}_i = t + P_k (x_i - t) $ where the centering is done with the robust location functional $t$ (instead of $t_0$) and the projection matrix $P_k$ is now taken to be onto the space spanned by the first $k$ eigenvectors of the robust scatter functional $S$ (instead of $S_0$). 
Due to its more technical nature in comparison to the other two criteria, we have postponed the discussion of the extension of $\hat{R}_{1k}$ to Section \ref{sec:derivative_extension}.

Finally, regardless of which of the three criteria $\hat{R}_{1k}$, $\hat{R}_{2k}$ and $\hat{R}_{3k}$ one uses, the corresponding estimate $\hat{d}$ of the latent dimension $d$ is obtained as the minimizing index,
\begin{align*}
    \hat{d} = \mathrm{argmin}_{k = 0, \ldots, p - 1} \hat{R}_{jk}.
\end{align*}
We next recall several popular options for the location-scatter pair $(t, S)$.

\subsection{Mean vector and covariance matrix}

The most typical choice for the pair $(t, S)$ is the mean vector and the covariance matrix, i.e.,
\begin{align}\label{eq:spatial_median_objective}
    t(F) = \mathrm{E}_F(x), \quad S(F) = \mathrm{E}_F [ \{x - t(F)\} \{x - t(F)\}' ],
\end{align}
where $\mathrm{E}_F$ means that the expectation is taken under the assumption that $x \sim F$. This choice simply leads to the Gaussian SURE-criterion discussed in Section \ref{sec:sure_pca}. As discussed before, this option is, despite often being the optimal choice under the assumption of normality, also highly non-tolerant against outliers and heavy tails.

\subsection{Spatial median and spatial sign covariance matrix}

The spatial median $t(F)$ of a distribution $F$ is defined as any minimizer of the convex function
\begin{align*}
    t \mapsto E_F \{ \| x - t \| - \| x \| \},
\end{align*}
over $t \in \mathbb{R}^p$. The spatial median is one of the oldest and most studied robust measures of multivariate location, see, \citep{haldane1948note, brown1983statistical}, and reverts to the univariate concept of median when $p = 1$. It can be shown to exist for any $F$ (in particular, no moment conditions are required) and it is unique as soon as $F$ is not concentrated on a line in $\mathbb{R}^p$ \citep{milasevic1987uniqueness} which is guaranteed, in particular, almost surely when $F$ is absolutely continuous.

The standard scatter functional counterpart for the spatial median is the spatial sign covariance matrix (SSCM), defined as,
\begin{align*}
    S(F) = \mathrm{E}_F \left\{  u( x - t(F) ) u( x - t(F) )' \right\},
\end{align*}
where $t(F)$ is the spatial median of $F$, which is assumed to be unique, and the sign function $u: \mathbb{R}^p \to \mathbb{R}^p $ is defined as $u(x) = x/\|x\|$ for $x \neq 0$ and $u(0) = 0$. Like its location counterpart, also the SSCM has been extensively studied in the literature, see, for instance, \citep{marden1999some, visuri2000sign, durre2016eigenvalues}. 

The defining feature of SSCM is that it depends on the data only through the ``signs'' $u(x - t(F))$, giving equal weight to points in a given direction regardless of their norm (which, in turn, is what makes the SSCM robust to outliers). Especially for high-dimensional data, this loss of information introduced by the discarding of the observation magnitudes is relatively small as it represents losing only a single degree of freedom in the $p$-dimensional space (whereas the sign contains the remaining $p - 1$ degrees of freedom).

We also note that the spatial median and the spatial sign covariance matrix are, strictly speaking, not a pair of location and scatter functionals in the usual sense as they satisfy the equivariance properties listed in Section \ref{sec:sure_pca} only when the matrix $A$ is orthogonal. However, this is not an issue in our scenario for the following reasons: (i) The spatial median is a consistent estimator of the location parameter in the elliptical model \eqref{eq:elliptical_model} under minor regularity conditions, see, e.g., \citep{magyar2011asymptotic}. (ii) Under the elliptical model \eqref{eq:elliptical_model}, two eigenvalues $\eigenvalue_j, \eigenvalue_\ell$ of the SSCM are equal if and only if the corresponding elements $\sigma_j, \sigma_\ell$ of the matrix $D$ are equal, see \citep{durre2016eigenvalues}. Hence, the eigenvalues of the spatial sign covariance matrix contain the same (qualitative) information about the latent signal dimension as those of any ``proper'', affine equivariant scatter functional. However, the spatial sign covariance matrix is also known to compress the range of the eigenvalues \citep{vogel2015robust}, making it more difficult to distinguish between the signal and the noise and, thus, we next consider an alternative to it. This alternative is known as Tyler's shape matrix and is often seen as the affine equivariant version of the SSCM.

\subsection{Tyler's shape matrix}

Tyler's shape matrix \citep{tyler1987distribution} is one of the earliest proposed and most studied scatter functionals, see, e.g., \citep{dumbgen2005breakdown, wiesel2012geodesic}. Using it requires a location functional $t$ and, in the following, we take this to be the spatial median, as is common in the literature. Tyler's shape matrix $S(F)$ is defined as any $S$ with $\mathrm{det}(S) = 1$ and satisfying the following fixed-point equation,
\begin{align}\label{eq:tyler_shape}
    \mathrm{E}_F [ u(S^{-1/2} \{x - t(F)\} ) u(S^{-1/2} \{x - t(F)\} )' ] = \frac{1}{p} I_p.
\end{align}
A unique solution $S(F)$ is obtained as soon as $F$ does not concentrate too heavily on a subspace in $\mathbb{R}^p$, see \citep{dumbgen2005breakdown} for the exact conditions. Inspection of the equation \eqref{eq:tyler_shape} also reveals that any solution $S$ to it is defined only up to its scale and, to obtain a unique representative, a popular choice is indeed to use the determinant condition $\mathrm{det}(S) = 1$ to fix the scale of solution, see \citep{paindaveine2008canonical}. Consequently, $S(F)$ does not describe the full scatter of $F$ but only its \textit{shape} (scale-standardized scatter). However, this is sufficient for our purposes as scaling preserves the ordering of the eigenvalues of $S(F)$ and, hence, their division into signal and noise. Note that this also means that Tyler's shape matrix satisfies the affine equivariance property discussed in the beginning of Section \ref{sec:example_functionals} only up to scale, $S(F_{A, b}) = \{ \mathrm{det}(A) \}^{-2/p} A S(F) A'$.

The computation of $S(F)$ can be shown to correspond to a geodesically convex minimization problem \citep{wiesel2012geodesic}, meaning that an efficient algorithm for its estimation in practice is straightforwardly constructed. In our simulations we have used the R-package \texttt{ICSNP} \citep{Ricsnp} for this purpose.

\subsection{Hettmansperger-Randles estimator} 

As our final choice for the location scatter pair $(t, S)$, we consider the so-called Hettmansperger-Randles (H-R) estimator, which was originally introduced in the context of robust location estimation (and the associated shape functional was obtained as a ``by-product'' of the location estimation) \citep{hettmansperger2002practical}. The H-R pair $(t(F), S(F))$ is defined as any $(t, S)$, with $\mathrm{det}(S) = 1$ and satisfying the following pair of fixed-point equations,
\begin{align}\label{eq:hr_part_2}
    \mathrm{E}_F \{ u(S^{-1/2} (x - t) ) \} = 0 \quad \quad \mbox{and} \quad \quad \mathrm{E}_F \{ u(S^{-1/2} (x - t) ) u(S^{-1/2} (x - t) )' \} = \frac{1}{p} I_p.
\end{align}
Observing that the LHS of the first equation in \eqref{eq:hr_part_2} is, disregarding the matrix $S^{-1/2}$, the gradient of the objective function \eqref{eq:spatial_median_objective} of the spatial median, we see that the H-R pair $(t(F), S(F))$ can be interpreted as simultaneously determined spatial median and Tyler's shape matrix. This concurrent estimation of location and scatter (or, rather, shape as again any solution $S$ to the fixed-point equations is unique at most up to scale) then makes the resulting estimator affine equivariant (up to scale in case of~$S$). 

Despite its attractiveness, the theoretical properties of the H-R estimator have garnered less attention in the literature when compared to its previously introduced alternatives. In particular, we are not aware of any studies investigating conditions that would guarantee the uniqueness of the solution~$(t, S)$.

To summarize, the three robust alternatives to the mean-covariance pair introduced in this section can all be seen to estimate analogous quantities, while at the same time forming a sort of ``hierarchy'' with respect to their equivariance properties: (i) the spatial median and SSCM satisfy affine equivariance only for orthogonal $A$, (ii) replacing SSCM with Tyler's shape matrix yields the full affine equivariance property for the scatter (shape) functional and, (iii) both the location and scatter (shape) components of the H-R estimate are affine equivariant. As affine equivariance is the natural transformation property for a scatter functional to have in the presence of the elliptical model \eqref{eq:elliptical_model}, we thus expect that the previous ordering applies also to the corresponding SURE-procedures' comparative performances in practice. This claim will be investigated through simulations in Section \ref{sec:simulations}.

\section{Robust extension of $\hat{R}_{1k}$}\label{sec:derivative_extension}

In this section, we explore extending the SURE-criterion $\hat{R}_{1k}$ in \eqref{eq:two_forms_for_sure} to accommodate an arbitrary location-scatter pair. The theoretical cost of such an extension is considerably larger than for $\hat{R}_{2k}$ and $\hat{R}_{3k}$ as, instead of simply plugging in eigenvalues, it involves computing the partial derivatives $(\partial/\partial x_{ij}) \hat{x}_{ij}$.

In the sequel, let the observed sample $x_1, \ldots , x_n$ of points in $\mathbb{R}^p$ be fixed and denote its empirical distribution by $F_n$. Moreover, for $\varepsilon > 0$ we let $F_{n, i, j, \varepsilon}$ denote the empirical distribution of the perturbed sample $x_1, \ldots, x_i + \varepsilon e_j, \ldots , x_n $ where $e_j$ is the $j$th vector in the canonical basis of $\mathbb{R}^p$. For the extension $\hat{R}_{1k}$ to be well-defined in the first place, $t$ and $S$ are naturally required to be differentiable in a suitable sense, and the next assumption formalizes this requirement.

\begin{assumption}\label{assu:derivatives}
For any $i = 1, \ldots, n$ and $j = 1, \ldots, p $, there exists $h_{ij} \in \mathbb{R}^p$ and a symmetric $H_{ij} \in \mathbb{R}^{p \times p}$ satisfying
\begin{align*}
     \frac{1}{\varepsilon} \{ t( F_{n, i, j, \varepsilon} ) - t( F_n ) \} \rightarrow h_{ij} \quad \quad \mbox{and} \quad \quad \frac{1}{\varepsilon} \{ S( F_{n, i, j, \varepsilon} ) - S( F_n ) \} \rightarrow H_{ij},
\end{align*}
as $\varepsilon \rightarrow 0$.
\end{assumption}

In order for also the projection matrix $P_k$ (onto the span of the first $k$ eigenvectors of $S(F_n)$) to be differentiable in the previous sense for all $k = 1, \ldots , p$, all eigenvalues of the matrix $S(F_n)$ must be simple. This condition, formalized in Assumption \ref{assu:eigenvalues} below, is rather mild and, in particular, holds almost surely for both the covariance matrix and the SSCM if the points $x_1, \ldots , x_n$ are drawn from an absolutely continuous distribution. 

\begin{assumption}\label{assu:eigenvalues}
 The eigenvalues of $S(F_n)$ are distinct.
\end{assumption}

Under the previous two assumptions, the partial derivatives $( \partial / \partial x_{ij} ) \hat{x}_{ij}$ exist and their sum over $j$ has the analytical form given in the next lemma.

\begin{lemma}\label{lem:reconstruction_derivative}
Under Assumptions \ref{assu:derivatives} and \ref{assu:eigenvalues}, we have
\begin{align*}
    \sum_{j=1}^p \frac{\partial}{\partial x_{ij}} \hat{x}_{ij} &= k + \sum_{j=1}^p \mathrm{tr}\{ (I_p - P_k) h_{ij} e_j' \} + \sum_{j=1}^p e_j' A_{ij} \{ x_i - t(F_n) \},
\end{align*}
where 
\begin{align*}
    A_{ij} := \sum_{\ell=1}^k \sum_{m = k + 1}^p \frac{1}{\eigenvalue_\ell - \eigenvalue_{m}} (T_\ell H_{ij} T_m + T_m H_{ij} T_\ell),
\end{align*}
$T_\ell$ is the orthogonal projection onto the space spanned by the $\ell$th eigenvector of $S(F_n)$, and $\eigenvalue_\ell$ is the corresponding eigenvalue.
\end{lemma}

Lemma \ref{lem:reconstruction_derivative} essentially says that, as soon as one obtains expressions for the quantities $h_{ij}$ and $H_{ij}$ in Assumption~\ref{assu:derivatives} (and Assumption~\ref{assu:eigenvalues} holds) for some particular location scatter pair $(t, S)$, these can be plugged in to Lemma 3 to construct a version of the SURE-criterion $\hat{R}_{1k}$ that is based on $(t, S)$. In Lemma \ref{lem:mean_covariance} below we have provided, for completeness, these expressions for the standard mean-covariance pair. The resulting SURE-criterion $\hat{R}_{1k}$ is, naturally, the Gaussian SURE as described in Section \ref{sec:sure_pca}.

\begin{lemma}\label{lem:mean_covariance}
The mean vector and the covariance matrix satisfy Assumption \ref{assu:derivatives} with
\begin{align*}
    h_{ij} = \frac{1}{n} e_j, \quad H_{ij} = \frac{1}{n} e_j \{ x_i - t(F_n) \}' + \frac{1}{n} \{ x_i - t(F_n) \} e_j'.
\end{align*}
\end{lemma}

Despite not offering us anything new, Lemma \ref{lem:mean_covariance} also serves in its simplicity as a contrast to our next result, detailing the forms of $h_{ij}$ and $H_{ij}$ for the spatial median/SSCM-pair. What makes deriving these quantities more complicated, compared to the mean-covariance pair, is the fact that no analytical expression is available for the spatial median (instead, it is obtained as a minimizer of the objective function described in Section \ref{sec:example_functionals}).

\begin{lemma}\label{lem:sm_sscm}
    Assume (i) that the points $x_1, \ldots , x_n$ are not concentrated on a line in $\mathbb{R}^p$ and, (ii) that $t(F_n) \neq x_i$, for all $i = 1, \ldots , n$. Then the spatial median and the spatial sign covariance matrix satisfy Assumption \ref{assu:derivatives} with
    \begin{align*}
        h_{ij} &= G^{-1} A_i e_j, \\
        H_{ij} &= \frac{1}{n} A_i e_j \frac{y_i'}{\| y_i \|} + \frac{1}{n} \frac{y_i}{\| y_i \|} e_j' A_i
        - \frac{1}{n} \sum_{\ell=1}^n A_\ell G^{-1} A_i e_j \frac{y_\ell'}{\| y_\ell \|}
        - \frac{1}{n} \sum_{\ell=1}^n \frac{y_\ell}{\| y_\ell \|} e_j' A_i G^{-1} A_\ell, 
    \end{align*}
    where $A_i := w_i (I_p - y_i y_i'/\| y_i \|^2)$, $w_i := \| y_i \|^{-1}$, $y_i := x_i - t(F_n)$ and $G := \sum_{i = 1}^n A_i$.
\end{lemma}

Plugging $h_{ij}$ and $H_{ij}$ from Lemma \ref{lem:sm_sscm} to the derivatives in Lemma \ref{lem:reconstruction_derivative} and consequently to $\hat{R}_{1k}$ in \eqref{eq:two_forms_for_sure} now gives us yet another robust criterion for determining the signal dimension. We note that while the additional assumption~(ii) imposed in Lemma \ref{lem:sm_sscm} seems to be difficult to analyze theoretically, its validity is nevertheless simply checked in practice (and the assumption (i) is satisfied, in particular, almost surely when $F$ is absolutely continuous).

Mimicking the proof of Lemma \ref{lem:sm_sscm} it would next be possible to derive equivalent results also for our two remaining location-scatter pairs. However, we have decided not to do so and the reasons for this are two-fold: (i) Some preliminary computations (not shown here) show that these computations lead, as with the spatial median/SSCM-pair in Lemma \ref{lem:sm_sscm}, to analytically cumbersome expressions for $h_{ij}$ and $H_{ij}$, from which no real insight can be gained. (ii) Due to the complexity of the resulting expression (and the large number of nested summations involved), the practical usefulness of the extensions is questionable. Indeed, as our timing comparisons in Section \ref{sec:timing} demonstrate, the version of $\hat{R}_{1k}$ obtained based on Lemma~\ref{lem:sm_sscm} is several orders of magnitude inferior to $\hat{R}_{2k}$ and $\hat{R}_{3k}$ in computational speed, while at the same time offering no or only minuscule gains in accuracy. Some preliminary exploration reveals that this issue is still further magnified for Tyler's shape matrix. Hence, while these extensions would be technically possible to derive, we did not see any real practical value in doing so.

\section{Simulations}\label{sec:simulations}

In order to study the finite-sample properties of our proposed robust extensions of SURE, we conduct an array of  simulation studies. As competing methods we have chosen the following set of well-established estimators from the literature, see Section \ref{subsec:scope} for further details:

\begin{itemize}
    \item[(i)] The classical estimator based on an asymptotic test of subsphericity \citep{schott2006high,nordhausen2021asymptotic}. The R-package \texttt{ICtest} \citep{RICtest} includes two implementations of it, one based on the covariance matrix and one based on the H-R estimator, and we include both of them in the comparison.
    \item[(ii)] The same estimator as (i) but with the null distribution of the test estimated through bootstrap. This estimator can be based on any of the four scatter matrices described in Section \ref{sec:example_functionals} and we thus include all of them in the comparison. We used throughout the study $200$ bootstrap samples, the default value in the implementation in \texttt{ICtest} \citep{RICtest}.
    \item[(iii)] The ladle estimator of \citep{luo2016combining} which, too, can be based on any of the four scatter matrices. The estimator is based on resampling, for which we used the default value $200$ in the implementation in \texttt{ICtest} \citep{RICtest}.
    \item[(iv)] Our centered version of the SURE-estimator of \citep{ulfarsson2015selecting}.
\end{itemize}

The above four categories of estimators are denoted in the following as {Asymp}, {Boot}, {Ladle} and SURE, respectively, with the used scatter matrix given in parenthesis. E.g., Asymp(HR) denotes the asymptotic test based method using the H-R estimator. In addition, we distinguish three different versions of the SURE-estimator, SURE1, SURE2 and SURE3, referring to using the objective functions $\hat{R}_{1k}$, $\hat{R}_{2k}$ and $\hat{R}_{3k}$, respectively. We thus have a total of 19 methods to compare, and these have been summarized in Table \ref{tab:methods}. The final four columns of the table are related to the timing study in Section \ref{sec:timing}. 

\begin{table}[]
    \centering
    \begin{tabular}{l|c|c|c|c|c}
  \multirow{2}{*}{Method(Scatter matrix)} & \multirow{2}{*}{Is robust?} & \multicolumn{2}{c}{$n = 200$} & \multicolumn{2}{c}{$n = 400$} \\
    & & $p = 10$ & $p = 20$ & $p = 10$ & $p = 20$ \\
    \hline
        SURE1(Cov) & No & 0.00  & 0.00  & 0.00  & 0.00   \\
        SURE1(SSCM) & Yes & 0.09  & 1.88  & 0.15  & 2.71   \\
        SURE2(SSCM) & Yes & 0.01  & 0.01  & 0.05  & 0.01  \\
        SURE2(Tyler) & Yes & 0.02  & 0.01  & 0.02  & 0.03   \\
        SURE2(HR) & Yes & 0.06  & 0.08  & 0.08  & 0.12   \\
        SURE3(Cov) & No & 0.00  & 0.00  & 0.00  & 0.00   \\
        SURE3(SSCM) & Yes & 0.01  & 0.01  & 0.02  & 0.02   \\
        SURE3(Tyler) & Yes & 0.01  & 0.02  & 0.03  & 0.03   \\
        SURE3(HR) & Yes & 0.06  & 0.09  & 0.08  & 0.12   \\
        \hline
        Asymp(Cov) & No & 0.00  & 0.00  & 0.00  & 0.00   \\
        Asymp(HR) & Yes & 0.05  & 0.07  & 0.09  & 0.13  \\
        \hline
        Boot(Cov) & No & 0.22  & 0.31  & 0.45  & 1.43  \\
        Boot(SSCM) & Yes & 2.19  & 5.35  & 4.07  & 10.36   \\
        Boot(Tyler) & Yes & 2.89  & 7.30  & 4.84  & 12.75  \\
        Boot(HR) & Yes & 8.89  & 24.14 & 13.30  & 37.04  \\
        \hline
        Ladle(Cov) & No & 0.08  & 0.14 & 0.09  & 0.16  \\
        Ladle(SSCM) & Yes & 1.88 & 2.14  & 3.16  & 3.69  \\
        Ladle(Tyler) & Yes & 2.96 & 4.24  & 4.43  & 6.03  \\
        Ladle(HR) & Yes & 17.47 & 26.81  & 18.36  & 27.37 
    \end{tabular}
    \caption{The estimators included in the simulation study, along with binary indicator for their robustness. The final four columns give the average running times (in seconds) of the methods over 10 replicates in various settings.}
    \label{tab:methods}    
\end{table}

Throughout the simulations, we measure the performance of the methods through their total proportions of correctly estimated dimensions across all replicates. 


\subsection{Tail thickness}

\label{sec:simu1}
In the first simulation study, we explore how the methods perform under varying levels of heavy-tailedness. As a setting for this, we consider multivariate $t$-distributions with the degrees of freedom equal to $\nu  = 1,3,5,\ldots,25$. Thus, the heaviest tails are obtained in the case $\nu = 1$, corresponding to the multivariate Cauchy distribution. The simulation is repeated 100 times for every degree of freedom, and for every repetition a random sample consisting of $n=100$ observations is generated. In each case, we take the latent dimension to be $d=6$, whereas as the total dimensionality we use $p=10$. The error ``variance'' (i.e., the square of the final diagonal elements of $D$ in \eqref{eq:elliptical_model}) is always $\sigma^2=0.5$ and the signal ``variances'' (i.e., the squares of the first $d$ diagonal elements of $D$ in \eqref{eq:elliptical_model}) are randomly generated from the uniform distribution $\texttt{Unif(1,3)}$, independently for each of the 100 repetitions. The proportions of correctly estimated dimensions $d$ for each of the 19 methods are presented in Figure \ref{fig:simu1}.

\begin{figure}
\includegraphics[width=1.0\columnwidth]{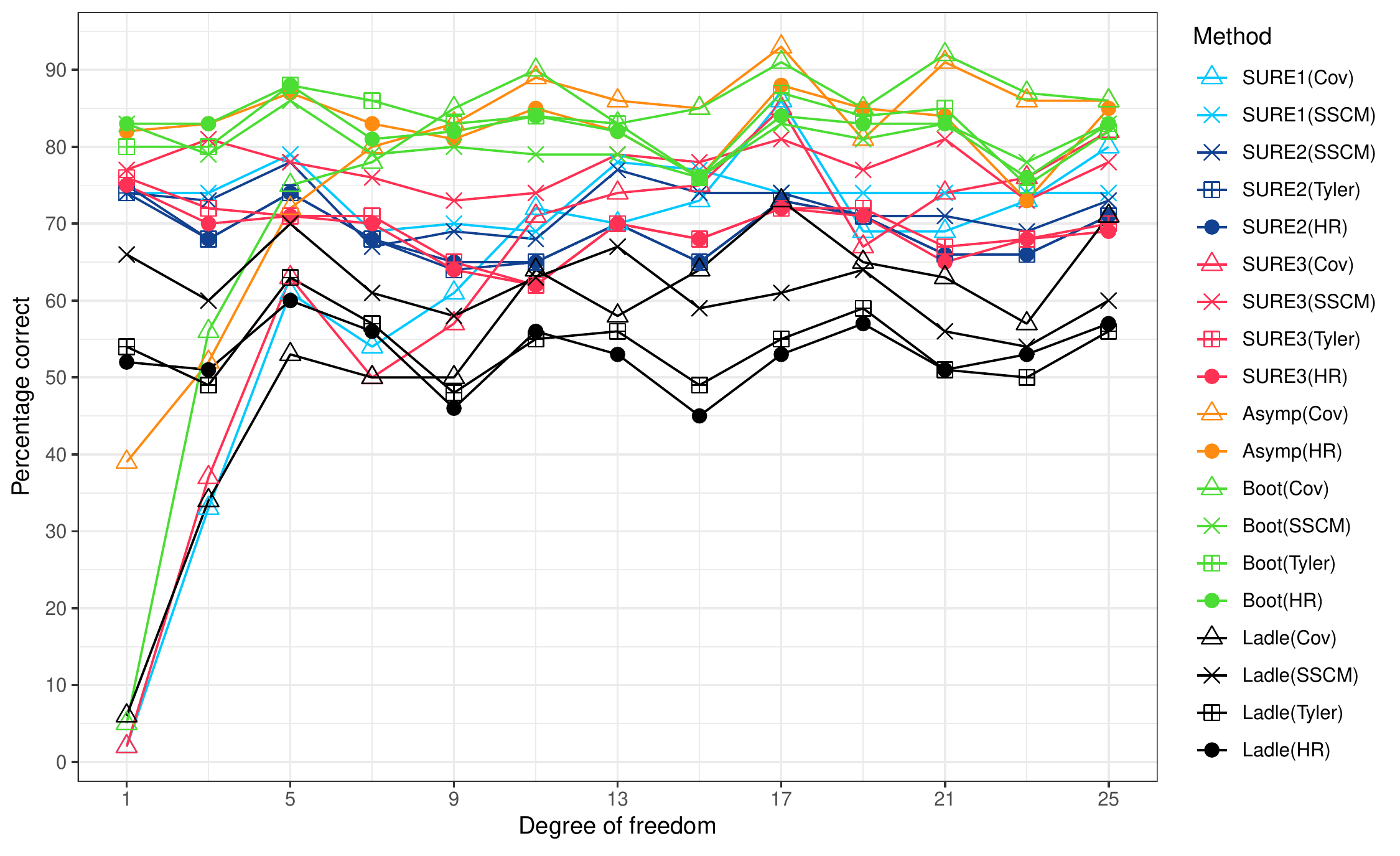}
\caption{Percentages of correctly estimated dimensions $d$ as a function of the degrees of freedom $\nu$ of the multivariate $t$-distribution in the tail thickness simulation. The sample size is $n = 100$ throughout.}
\label{fig:simu1}
\end{figure}

Unsurprisingly, the classical covariance matrix based methods fail to consistently find the correct latent dimension $d$ in the presence of too heavy tails (left side of the plot). This effect is most pronounced in the case $\nu = 1$ where the corresponding $t$-distribution does not possess the finite second-order moments required by the covariance estimation. The robust methods, on the other hand, do not suffer from this issue as they make no moment assumptions on the data generating distribution.

As the degrees of freedom increase, we observe that the covariance based methods start to outperform the robust alternatives. The reason for this is that, when $\nu \rightarrow \infty$, the multivariate $t$-distribution approaches the normal distribution for which the covariance based methods offer optimal inference. 

Comparing the methods by type (different colours in Figure \ref{fig:simu1}), the overall best performances are given by the robust bootstrap (Boot, green) and asymptotic (Asymp, orange) methods, with marginal differences between the two. Somewhat surprisingly, the computationally most expensive method, i.e., the ladle, has a relatively bad performance in this scenario. Comparing the different types of SURE to each other, it seems that the additional computational and theoretical complexity of SURE1 (cyan) does not provide additional benefits when compared to SURE2 (blue) and SURE3 (red), which both have a relatively similar performance, not falling much behind Boot and Asymp. 

\subsection{Latent dimension}\label{sec:simu2}

In our second simulation study, we investigate how the relative size of the underlying latent dimension $d$ affects the methods' performances. As the main selling point of the SURE-based methods is their light computational load, we have dropped the more computationally intensive methods (Boot, Ladle) from the comparison, focusing in this (and the following) study on comparing SURE only to its most relevant competitor, Asymp. We choose SURE2 to be the ``representative'' of the SURE-family as, based on the first simulation study, both SURE1 and SURE3 had performance similar to it. Thus, the families of methods included in the current simulation study are Asymp and SURE2.

Recall that SURE2 estimates the latent dimension as the index minimizing the corresponding objective function $k \mapsto \hat{R}_{2k}$. Based on our experiments, this strategy can sometimes be quite unstable, especially when the latent dimension is comparatively small. Thus, as an experimental alternative we propose estimating $d$ as the change point in the series of differences $\hat{R}_{2(k + 1)} -  \hat{R}_{2k}$. To understand the motivation for this, consider the following two typical forms for the graph formed by the points $(k, \hat{R}_{2k})$: (i) The points $(k, \hat{R}_{2k})$ form a V-shaped curve around $d$. In this case, the true dimension is both a minimizer and a location change point of the differences (the differences change sign at the true dimension). (ii) The graph $(k, \hat{R}_{2k})$ decreases linearly until $d$ and stays roughly constant afterwards. In this case, $d$ is a location change point of the differences but not necessarily a minimizer (it might happen that the minimizer occurs only after $d$). Thus, in these two (rather idealistic) examples, the experimental change point alternative offers more consistent detection of the dimension than the standard method of seeking the minimizer. We implemented the change point detection as binary segmentation through the function \texttt{cpt.meanvar} in the R-package \texttt{changepoint}~\cite{killick2014changepoint}. The resulting method is denoted in the sequel as ``SURE2 cp''.

We fix the total dimensionality to $p=100$ and let the latent dimension vary as $d=5,10,15,\ldots,95$. We consider two sample sizes $n = 1000, 2000$ and repeat the simulation 100 times for every combination of $d$ and $n$, generating in each repetition a random sample from the multivariate $t$-distribution with 1 degree of freedom. The error variance is fixed to $\sigma^2=0.5$ and the signal variances are randomly generated from the uniform distribution $\texttt{Unif(1,3)}$, independently in every repetition. The proportions of correctly estimated dimensions $d$ for the different procedures are presented in the two panels of Figure \ref{fig:simu2}, separately for $n = 1000$ and $n = 2000$.


\begin{figure}
\includegraphics[width=1.0\columnwidth]{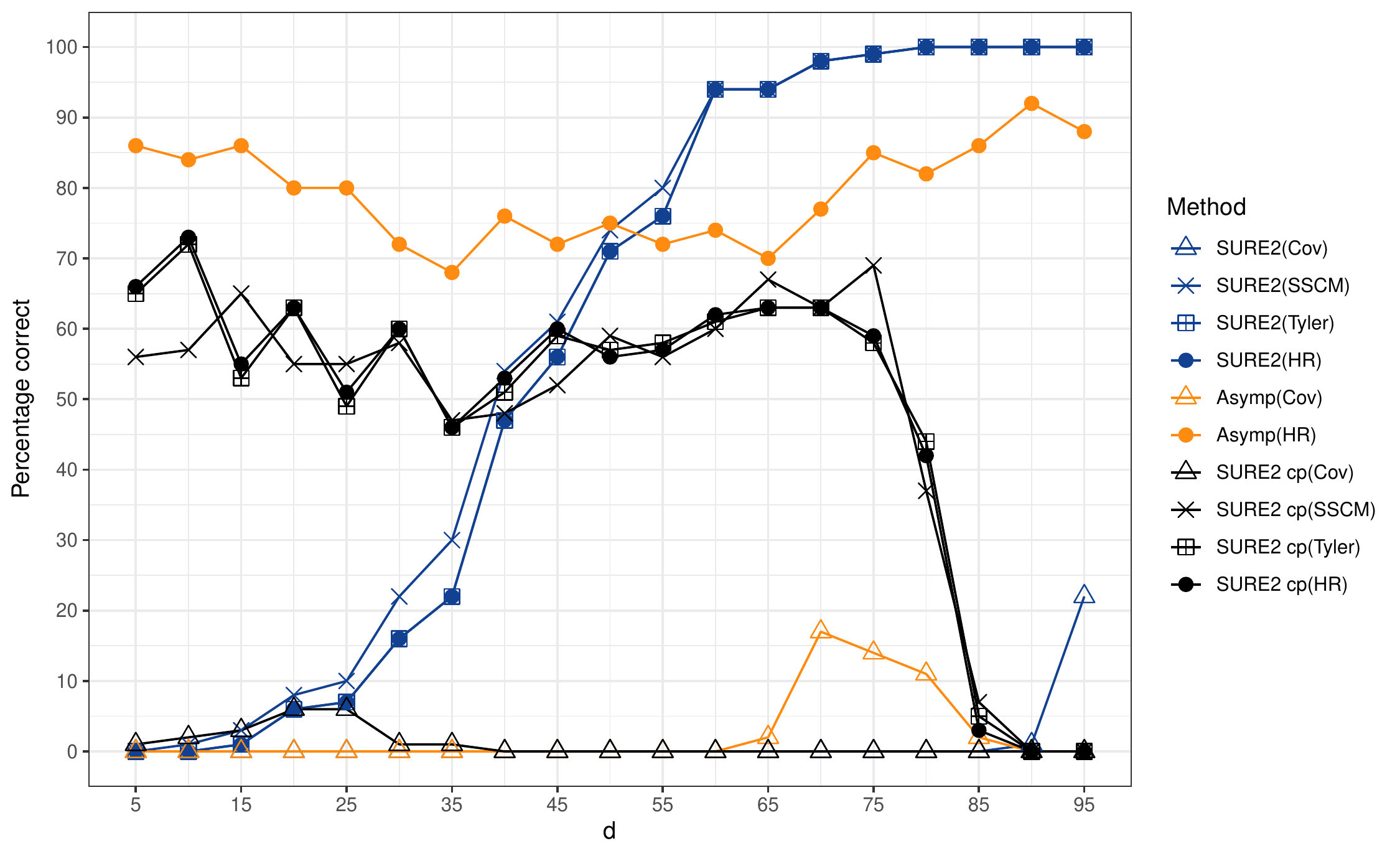}
\includegraphics[width=1.0\columnwidth]{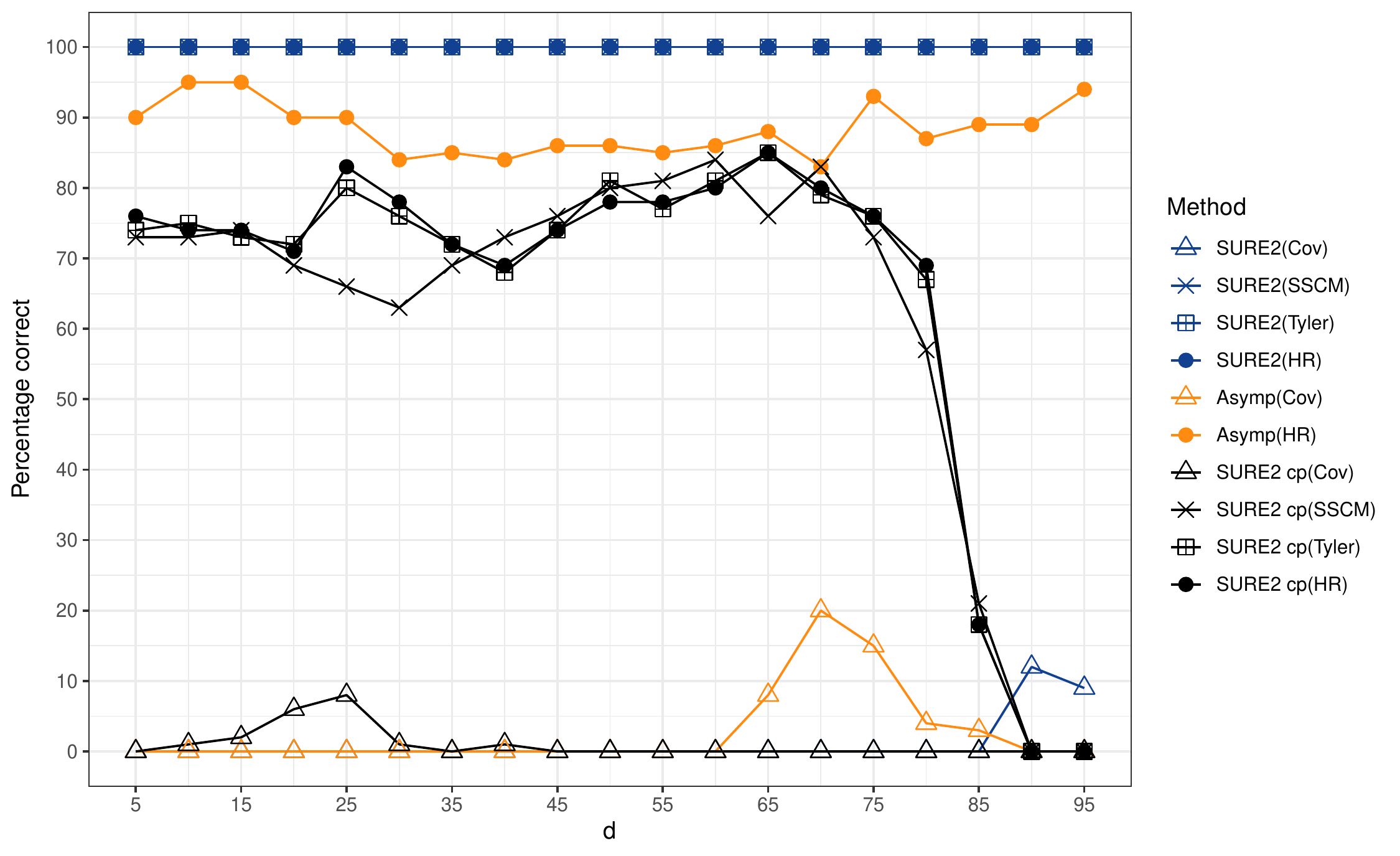}
\caption{Percentages of correctly estimated dimensions as a function of the underlying dimension $d$ when sample size is $n = 1000$ (top panel) or $n = 2000$ (bottom panel).}
\label{fig:simu2}
\end{figure}

The covariance based variants of Asymp, SURE2 and SURE2 cp have the expected (bad) performance, and in the following discussion we focus only on their robust versions. The top panel of Figure \ref{fig:simu2} reveals that, when $n=1000$, SURE2 starts producing good results only once the latent dimensionality $d$ is roughly half the data dimension $p$. However, the doubling of the sample size from $n = 1000$ to $n = 2000$ has a drastic effect on its performance and, under the latter, SURE2 estimates the true dimension correctly in each of the 100 repetitions (and for every value of $d$). SURE2 cp, on the other hand, has the opposite performance for $n = 1000$, giving the most accurate results when $d$ is relatively low. However, it does not benefit as much from the increased sample size as SURE2 does. Neither of the SURE-methods appears to be much affected by the actual choice of the robust scatter matrix (SSCM, Tyler or HR). Asymp is seen to give a consistently accurate behaviour under all considered settings and, as such, offers a reasonably safe choice in practice.



\subsection{Sample size}

In the third simulation, we study the effect of the sample size on the estimation accuracy, including again the same set of methods as in the previous study. The considered sample sizes are $n=500,750,1000,1500,2000,2500,5000$. The simulation is repeated 100 times for every sample size $n$, such that for every repetition a random sample of $n$ observations is generated from the multivariate $t$-distribution with 1 degree of freedom. We take the latent and the total dimensionalities to be $d=20$ and $p=100$, respectively, throughout the simulation. As the error variance we use $\sigma^2=0.5$ and the signal variances are again randomly generated from the uniform distribution $\texttt{Unif(1,3)}$, independently for each of the 100 replicates. The proportions of correctly estimated dimensions $d$ for the different procedures are presented in Figure \ref{fig:simu3}.

\begin{figure}
\includegraphics[width=1.0\columnwidth]{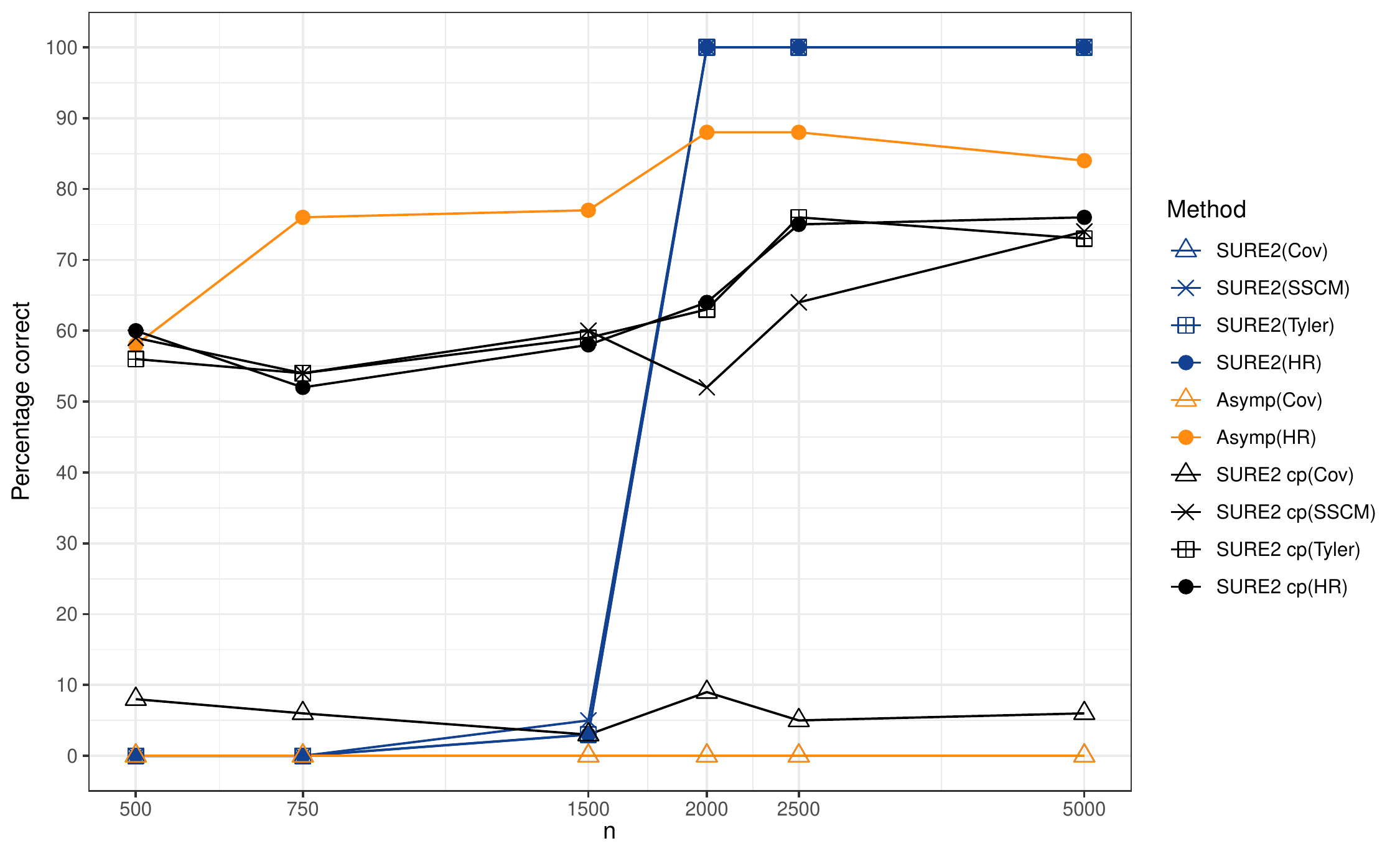}
\caption{Percentages of correctly estimated dimensions as a function of the sample size $n$. The scale of the horizontal axis is logarithmic.}
\label{fig:simu3}
\end{figure}

The most striking feature of Figure \ref{fig:simu3}, which also sheds some light on the behaviour of SURE2 in the previous simulations, is the sudden jump in the estimation accuracy of robust SURE2 from close to 0\% to 100\% at around $n = 1700$. Interestingly, the fact that SURE2 cp has more even behaviour across the different sample sizes indicates that this jump is not so much a consequence of the SURE criterion $\hat{R}_{2k}$ itself but of the way in which the dimension estimate is selected based on the criterion values (recall that SURE2 picks the minimizing value of $k$ and SURE2 cp uses a more complicated change point technique). We thus conclude that the standard technique of choosing simply the minimizing index of the SURE criterion is not optimal at smaller sample sizes (in this scenario), whereas it gets very effective when $n$ is large enough. This matter clearly warrants further investigation and, due to its complexity, we have left it for future research, see Section \ref{sec:discussion}. Finally, we also observe that, overall, the used scatter matrix seems to have very little effect on the results, apart from Cov, which again breaks down in the presence of a heavy-tailed distribution. 


\subsection{Computation time}\label{sec:timing}

As our final simulation study, we compare the running times of all 19 methods included in Table \ref{tab:methods}. The change point variant of SURE2 is not included as its computational difference to the base SURE2-method is marginal and negligible compared to the differences between the methods itself. We distinguish two different sample sizes $n = 200, 400$ and two different dimensionalities $p = 10, 20$, their combinations leading to a total of four different settings. For each setting, we take the data distribution to be the multivariate $t$-distribution with $\nu = 1$ degrees of freedom, $d = 0.6 p$ and the signal and noise variances as in the previous simulation. We run each of the 19 methods 10 times on each setting (using the same set of 10 data for each method) and record their computational times. The experiment was conducted on a desktop computer with AMD Ryzen 5 3600 6-core processor and 16 GB RAM. The average running times in seconds are given in the final four columns of Table \ref{tab:methods}.

From Table \ref{tab:methods} we make the following observations: (i) Computational complexity in the methods stems from two sources, the choice of the scatter matrix and bootstrap replications (as performed by Boot and Ladle), of which the latter has a significantly greater impact on the timing. (ii) The doubling of the sample size $n$ does very little to increase the computational times, whereas the doubling of the dimension $p$ serves to multiply the times by roughly 1.5. (iii) Of the robust methods, the fastest are by far SURE2, SURE3 and Asymp which all have computational times roughly of the same order of magnitude. SURE1 falls somewhere in between them and the more intensive Boot and Ladle.

Based on the observations made above and in the previous experiments, we summarize our simulation results as follows: The SURE-based robust methodology offers a fast alternative to the computationally heavy bootstrap-based methods, losing minimally to them in accuracy but at the same time retaining its usefulness also in the face of high-dimensional data sets. When compared to Asymp, SURE2 loses to it in accuracy when both $n$ and the latent dimension $d$ are small but quickly surpasses Asymp when $n$ is increased sufficiently, at least in the scenario we considered. Additionally, whereas Asymp is (for theoretical reasons) restricted to using the computationally heavy H-R estimator to achieve robustness, SURE2 does not have this limitation and we may replace the H-R estimator in SURE2 with, e.g., SSCM to obtain roughly a ten-fold gain in speed, based on Table \ref{tab:methods}, with almost identical accuracy. We thus conclude that the SURE-based robust methods offer fast and efficient estimators of the signal dimension in large-dimensional data-rich scenarios.

\section{Application: asset returns}\label{sec:real_data}

\begin{figure}
    \centering
    \includegraphics[width=1\textwidth]{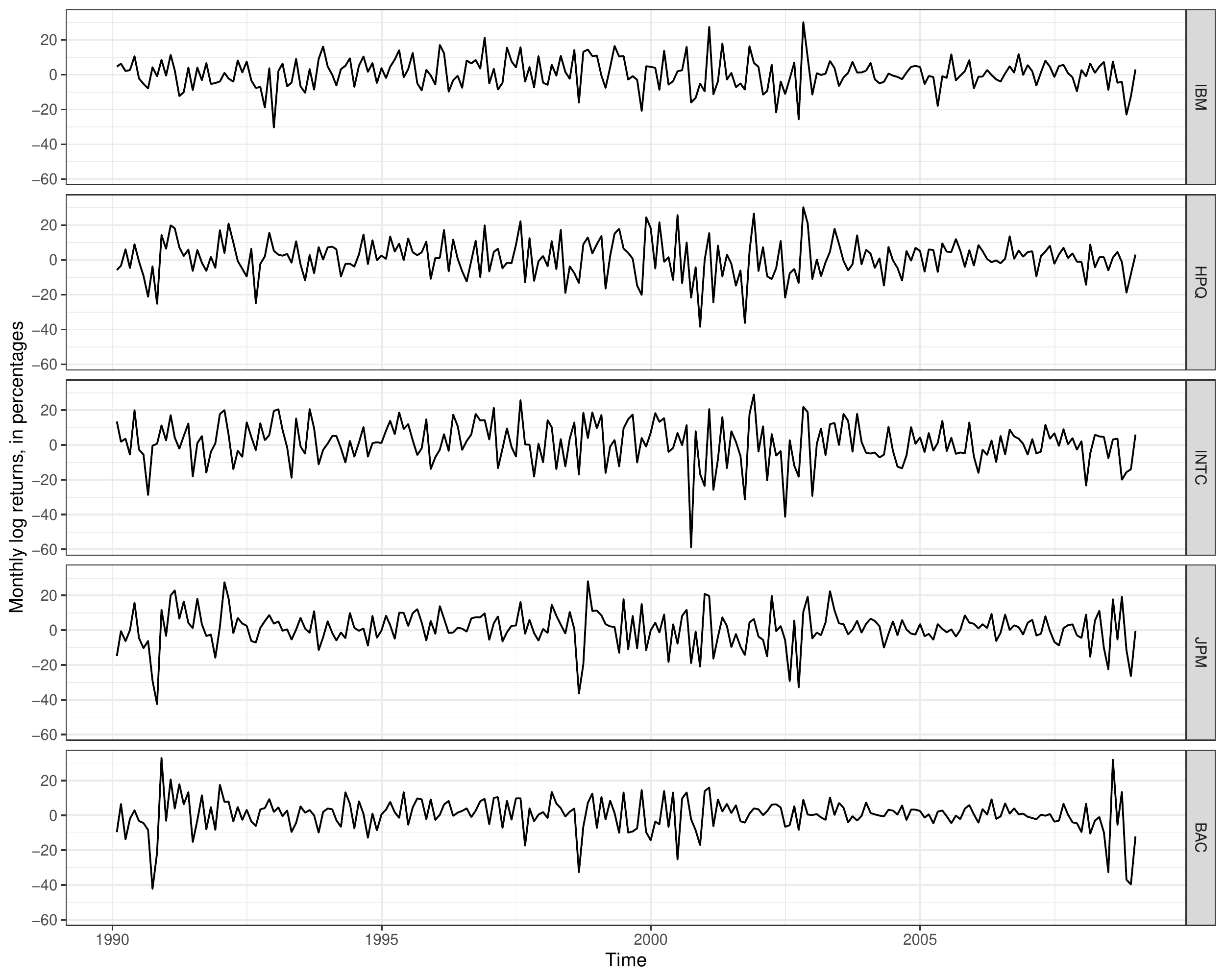}
    \caption{The monthly log returns (in percentages) of five stocks (IBM, HPQ, INTC, JPM, BAC) from January 1990 to December 2008.}
    \label{fig:tsay_original}
\end{figure}

We next illustrate the robust SURE methods in a financial data set that was used to demonstrate principal component analysis in the classical textbook \cite[Section 9]{tsay2010analysis} to search for common factors explaining (joint) asset return variability. The data itself is available on the author's (Ruey S. Tsay's) webpage and consists of monthly log stock returns (including dividends) of five stocks (IBM, HPQ, INTC, JPM, BAC) from January 1990 to December 2008.\footnote{See the website \texttt{https://faculty.chicagobooth.edu/ruey-s-tsay/research/analysis-\\of-financial-time-series-3rd-edition}. Specifically, stocks are International Business Machines (IBM), Hewlett-Packard (HPQ), Intel Corporation (INTC), J.P. Morgan Chase (JPM) and Bank of America (BAC).} These $p = 5$ time series of length $T = 228$ months are illustrated in Figure \ref{fig:tsay_original}. \cite{tsay2010analysis} computed the portmanteau test statistics and found that despite the time series nature of the data there is no substantial serial correlation in returns, and hence we also ignore the serial dependence in our analysis. According to the results of PCA, \cite{tsay2010analysis} concluded two common factors in his interpretation: The ``market component'' represents the general movement of the stock market and the ``industrial component'' represents the difference between the two industrial sectors, namely technology (IBM, HPQ and INTC) versus financial services (JPM and BAC). In addition, \cite{tsay2010analysis} points out that IBM stock ``has its own features that are worth further investigation''.

\begin{figure}
    \centering
    \includegraphics[width=1\textwidth]{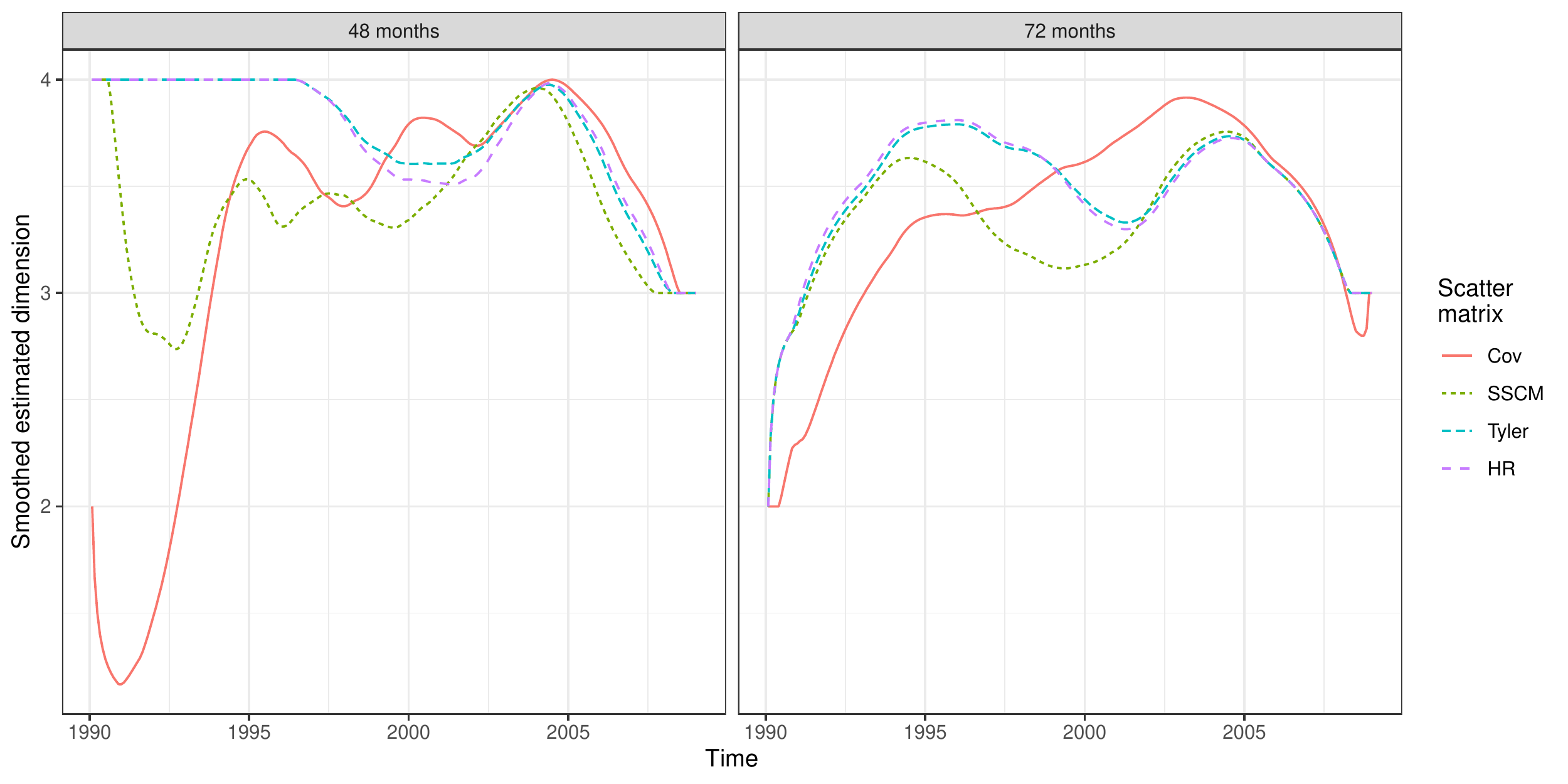}
    \caption{The smoothed curves of latent dimensionalities estimated using the window approach as described in the main text. The two panels correspond to the window lengths of $\ell = 48$ and $\ell = 72$ months, respectively.}
    \label{fig:tsay_dimensions}
\end{figure}

Instead of computing only a single estimate of the latent dimension for the data set, we take a local approach and run a window of length $\ell$ through the data. For each of the $T - \ell + 1$ windows we then estimate the latent dimension using one of our proposed methods. As the obtained dimensionalities correspond to the windows and not the actual observation months, we ``back-transform'' them as follows: For each individual month, we take the weighted average of the estimated dimensions of all windows in which that particular month is a member. We assign the weights such that the middle two observations in each window get a weight equal to 1 and the weights decrease linearly towards the window endpoints. This procedure thus produces a ``smoothed'' curve of estimated latent dimensions for the full observation period. To guarantee that the ellipticity of the data is at least partially fulfilled, we resort to rather small window lengths, taking either $\ell = 48$ or $\ell = 72$ in the following. Visual inspection (not shown here) reveals that scatter plots of the obtained windows indeed exhibit elliptical shapes throughout the observation period. The intuition behind this somewhat experimental approach is that the latent dimension $d$ can be seen to measure the internal complexity of the observed multivariate time series, allowing us to identify from the smoothed curve intervals of time when the stocks behave in a more unified manner (low dimension) or independently of each other (high dimension). Note that the months close to the beginning and the end of the measurement interval belong on average to a fewer number of windows, meaning that they are expected to show more erratic behavior.

For simplicity, we apply the described procedure only with SURE2, using each of the four scatter matrices (and the two window lengths), and the results are shown in Figure \ref{fig:tsay_dimensions}. In line with the evidence in \cite{tsay2010analysis}, we find at least two and in most time points at least three common features in our five stock case. Our robust approaches favour three (even four) dimensions, emphasizing also other common features than the market and industrial components. Interestingly, the time-varying patterns of the estimated dimensionality in the robust approaches seem to be largely in accordance with each other and following general market conditions where the major periods of decreasing prices (i.e. the beginning of 2000s and the financial crisis 2007--2009), and their onsets, are associated with somewhat lower dimensions. Finally, note that the fact that the curves for the non-robust ``Cov'' are markedly different from the robust ones is a clear indication that the data indeed exhibit heavy-tailed behaviour (as typically with asset returns) that hampers the estimation of the covariance matrix but leaves the robust estimates unaffected.

\section{Discussion}\label{sec:discussion}

The results obtained in this work open up several avenues for future research. Perhaps most notably, our simulations revealed that the standard approach of choosing the parameter estimate to be the global minimizer of the SURE-criterion might not be optimal in the current scenario. As a simple alternative to minimization, we explored using a change point detection based approach which indeed proved superior to the minimization in various settings (and vice versa in other settings). As such, the matter clearly warrants more investigation. We also note that the dangers of ``blindly'' minimizing a model fit criterion are naturally well-known in the model selection literature. Indeed, also \cite{ulfarsson2008dimension} mention this caveat, although, not in the context of SURE but the Bayesian information criterion. Despite this, we are not aware of any general alternatives to minimization being proposed in the literature, discounting visual inspection (which can be seen as both heuristic and subjective). Quite possibly, such procedures may not even exist as their behaviour would depend greatly on the functional form of the particular criterion in question. But, at least in the current scenario of dimension estimation, the change point criterion appears to provide a feasible option. 

A second point of interest brought up by our simulations is the sudden increase in the accuracy of the robust SURE2 in Figure \ref{fig:simu3}. In that particular data scenario there appears to be a ``critical'' sample size after which SURE2 achieves perfect estimation results. The dependence of this sample size on the model parameters, especially $p$ and $d$ is something to be studied and quantified. We note that any theoretical investigation of this matter is likely to be very difficult as it concerns the finite-sample properties of the method, whereas the large majority of the existing theoretical results in the robust literature are conducted in the asymptotic framework where $n \rightarrow \infty$.

Finally, despite being a significant improvement over the Gaussian assumption, the elliptical model \eqref{eq:elliptical_model} can still be seen as somewhat limiting in practice. In particular, a consequence of the ellipticity is that the tails of the distribution are equally heavy along all directions in the space $\mathbb{R}^p$. One solution would be to consider the independent component model $x_i = \mu + A z_i$ instead, where $A \in \mathbb{R}^{p \times p}$ is full rank and the components of the $p$-vector $z_i$ are mutually independent random variables \citep{comon2010handbook}. Exactly $p - d$ elements of $z_i$ are assumed to be Gaussian and, as such, noise, making the signal dimension of the model equal to $d$. By taking the signal components of $z_i$ to have different tail decays, a richer variety of heavy-tailed behaviours for $x_i$ is obtained, see \cite{virta2020latent}. The independent component model admits a solution through the use of \textit{pairs} of scatter functionals, see \cite{tyler2009invariant}, and a similar approach could possibly serve as a starting point for deriving a SURE-based criterion for the estimation of the dimension $d$ in the independent component model.

\section*{Acknowledgements}
The work of Niko Lietz\'en and Henri Nyberg was supported by the Academy of Finland (grant 321968). The work of Joni Virta was supported by the Academy of Finland (grant 335077). Nyberg also acknowledges financial support by the Emil Aaltonen Foundation (grant 180287).


\appendix

\section{Proofs of the technical results}\label{sec:proofs}

\begin{proof}[Proof of Lemma \ref{lem:sure_pca}]

We write, 
\begin{align*}
    & \mathrm{E} \| \hat{x}_i - V_0 y_i \|^2\\
    =& \mathrm{E} \| (\hat{x}_i - x_i) + (x_i - V_0 y_i) \|^2 \\
    =& \mathrm{E} \| \hat{x}_i - x_i \|^2 + 2 \mathrm{E} \{ (\hat{x}_i - x_i)'\varepsilon_i \} +  \mathrm{E} \| \varepsilon_i \|^2.
\end{align*}
Now, $ \| \hat{x}_i - x_i \|^2 = \mathrm{tr} \{ (I_p - P_k) (x_i - t_0) (x_i - t_0)' \} $, implying that averaging over $i$ in the preceding expansion yields
\begin{align*}
    & \frac{1}{n} \sum_{i=1}^n \mathrm{E} \| \hat{x}_i - V_0 y_i \|^2 = \mathrm{E} [ \mathrm{tr} \{ ( I_p - P_k ) S_0 \} ] + \frac{2}{n} \sum_{i=1}^n \mathrm{E} ( \hat{x}_i'\varepsilon_i ) - p \sigma^2.
\end{align*}
Hence, the claim follows as soon as we show that $\mathrm{E} ( \hat{x}_{ij} \varepsilon_{ij} ) =  \sigma^2 \mathrm{E} \{  ( \partial / \partial x_{ij} ) \hat{x}_{ij} \}$. To see this, we use the law of total expectation in conjunction with Stein's lemma to write,
\begin{align*}
    \mathrm{E} ( \hat{x}_{ij} \varepsilon_{ij} ) =& \mathrm{E} \{ \mathrm{E} ( \hat{x}_{ij} \varepsilon_{ij} \mid y_1, \ldots , y_n ) \} \\
    =& \sigma^2 \mathrm{E} [ \mathrm{E} \{ ( \partial / \partial \varepsilon_{ij} ) \hat{x}_{ij} \mid y_1, \ldots , y_n \} ] \\
    =& \sigma^2 \mathrm{E} [ \mathrm{E} \{ ( \partial / \partial x_{ij} ) \hat{x}_{ij} \mid y_1, \ldots , y_n \} ] \\
    =& \sigma^2 \mathrm{E} \{  ( \partial / \partial x_{ij} ) \hat{x}_{ij} \},
\end{align*}
where the second equality uses Stein's lemma on the reconstruction $\hat{x}_{ij}$ treated as a function of the full data (which are, conditional on $y_1, \ldots , y_n$, Gaussian), the third equality uses the fact that $x_{ij}$ and $\varepsilon_{ij}$ are equal up to translation by a constant (again, under the conditioning). 
\end{proof}

\begin{proof}[Proof of Lemma \ref{lem:reconstruction_derivative}]
We use, for brevity, the notation $t := t(F_n)$, $t^* := t( F_{n, i, j, \varepsilon} )$, and similarly for other quantities. Hence, by Assumption \ref{assu:derivatives}, we have $t^* = t + \varepsilon h_{ij} + o(\varepsilon)$ and $S^* = S + \varepsilon H_{ij} + o(\varepsilon)$. Thus, by \citep[Chapter 1.3.2]{stewart2001matrix}, the matrix $S^*$ has a $\ell$th eigenvector $u_\ell^*$ with the following first-order expansion,
\begin{align}\label{eq:eigenvector_perturbation}
    u_\ell^* = u_\ell + \varepsilon \sum_{m \neq \ell}^p \frac{1}{\eigenvalue_\ell - \eigenvalue_m} T_m H_{ij} u_\ell + o(\varepsilon),
\end{align}
where $u_\ell$ is the $\ell$th eigenvector of $S$ and $T_\ell = u_\ell u_\ell'$ is the corresponding orthogonal projection. Taking outer products on both side of \eqref{eq:eigenvector_perturbation} gives
\begin{align*}
    T_\ell^* = T_\ell + \varepsilon \sum_{m \neq \ell}^p \frac{1}{\eigenvalue_\ell - \eigenvalue_m} (T_\ell H_{ij} T_m + T_m H_{ij} T_\ell) + o(\varepsilon).
\end{align*}
Consequently, the projection $P^*_k$ onto the space spanned by the eigenvectors corresponding to the $k$ largest eigenvalues of $S^*$ (which are, for small enough $\varepsilon$, simple) is $P_k^* = P_k + \varepsilon A_{ij} + o(\varepsilon)$, where $A_{ij}$ is as in the statement of the lemma (note that, by symmetry of the $H_{ij}$, all terms with $ m \leq k $ get cancelled).

Now, recall that $x_i^* = x_i + \varepsilon e_j$ and write,
\begin{align*}
    \hat{x}_i^* - \hat{x}_i &= t^* - t + P_k^* (x_i^* - t^*) - P_k (x_i - t) = \varepsilon \{ h_{ij} + P_k (e_j - h_{ij}) + A_{ij} (x_i - t) \} + o(\varepsilon),
\end{align*}
showing that the desired derivative equals
\begin{align*}
    \frac{\partial}{\partial x_{ij}} \hat{x}_{ij} = e_j'h_{ij} + e_j' P_k (e_j - h_{ij}) + e_j' A_{ij} (x_i - t).
\end{align*}
Summation over $j$ now yields the claim.
\end{proof}

\begin{proof}[Proof of Lemma \ref{lem:mean_covariance}]
We write $t := t(F_n)$ and $t^* := t(F_{n, i, j, \varepsilon})$, and similarly for $S$ aad $S^*$. Clearly, $t^* = t + (1/n) \varepsilon e_j$, implying that $x_i + \varepsilon e_j - t^* = x_i - t + (1 - 1/n) \varepsilon e_j$. Writing $y_i := x_i - t$, we get,
\begin{align}\label{eq:covariance_expansion}
    S^* =& \frac{1}{n} \sum_{\ell \neq i}^n \{ y_\ell - (1/n) \varepsilon e_j \} \{ y_\ell - (1/n) \varepsilon e_j \}' + \frac{1}{n} \{ y_i + (1 - 1/n) \varepsilon e_j \} \{ y_i + (1 - 1/n) \varepsilon e_j \}'.
\end{align}
The first term of \eqref{eq:covariance_expansion} simplifies to
\begin{align*}
    S - \frac{1}{n} y_i y_i' + \frac{1}{n^2} \varepsilon e_j y_i' + \frac{1}{n^2} \varepsilon y_i e_j' + o(\varepsilon),
\end{align*}
whereas the second term can be written as
\begin{align*}
    \frac{1}{n} y_i y_i' + \frac{n - 1}{n^2} \varepsilon e_j y_i' + \frac{n - 1}{n^2} \varepsilon y_i e_j' + o(\varepsilon).
\end{align*}
Combining these two now yields the claim.
\end{proof}

\begin{proof}[Proof of Lemma \ref{lem:sm_sscm}]
We begin with $h_{ij}$. Let $g:\mathbb{R}^p \to \mathbb{R} $ be the objective function with $g(t) = (1/n) \sum_{i=1}^n \| x_i - t \|$, from whose minimization the spatial median $t(F_n)$ is found. It is straightforwardly checked that $g$ is convex and, since, by the assumption that $t(F_n) \neq x_i$ for all $i = 1, \ldots , n$, $g$ is differentiable in a neighbourhood of $t(F_n)$, the gradient of $g$ must vanish at $t(F_n)$,
\begin{align}\label{eq:gradient_condition}
    \frac{1}{n} \sum_{i = 1}^n \frac{x_i - t(F_n)}{\| x_i - t(F_n) \|} = 0.
\end{align}

Define now the function $f: \mathbb{R}^{p + 1} \to \mathbb{R}^p $ such that
\begin{align*}
    f(\varepsilon, t) = \frac{1}{n} \sum_{\ell \neq i}^n \frac{x_i - t}{\| x_i - t \|} + \frac{1}{n} \frac{x_i + \varepsilon e_j  - t}{\| x_i + \varepsilon e_j - t \|}.
\end{align*}
Now, $f\{ 0, t(F_n) \} = 0$ and, moreover, $f$ is continuously differentiable at $(0, t(F_n))$ and has the Jacobian $D_t f \{ 0, t(F_n) \} = - (1/n) G$ where $G$ is as defined in the statement of the lemma. Now, $G$ is a sum of weighted projection matrices to the orthogonal complements of the lines spanned by the vectors $y_i = x_i - t(F_n)$. Hence, $G$ has full rank (by our assumption that the points $x_1, \ldots , x_n$ are not concentrated on a line) and the implicit function theorem then guarantees that for a suitably small neighbourhood $\mathcal{S} \subset \mathbb{R}$ of zero there exists a differentiable function $b: \mathcal{S} \rightarrow \mathbb{R}^p$ such that $f\{ \varepsilon, b(\varepsilon ) \} = 0 $.

In the following, with some abuse of notation, $\mathcal{S}$ will denote a (changing) small enough neighbourhood of zero in $\mathbb{R}$. By our sample being not concentrated on a line and \citep{milasevic1987uniqueness}, the spatial median of $F_{n, i, j, \varepsilon}$ is unique in $\mathcal{S}$. Let $M := \max_i \| x_i \|$ and take any $t \in \mathbb{R}^p$ such that $\| t \| \geq 3 M + 1$. Then, $(1/n)\sum_{i = 1}^n \| x_i - t \| \geq 2M + 1$, whereas $(1/n)\sum_{i = 1}^n \| x_i - 0 \| \leq M$, showing that, for all $\varepsilon \in \mathcal{S}$, the spatial medians $t(F_{n, i, j, \varepsilon})$ reside in a compact set. Thus, by Berge's Maximum Theorem, the map $\varepsilon \mapsto t(F_{n, i, j, \varepsilon})$ is continuous in $\mathcal{S}$, implying that the equivalent of the assumption that $t(F_n) \neq x_i$, for all $i = 1, \ldots , n$, holds for $F_{n, i, j, \varepsilon}$. Consequently, for any $\varepsilon \in \mathcal{S}$, the corresponding spatial median objective function is differentiable and its gradient vanishes at the spatial median $t(F_{n, i, j, \varepsilon})$. Its gradient is equal to $f ( \varepsilon, t )$ and, by the uniqueness of the spatial medians $t(F_{n, i, j, \varepsilon})$, we thus have that the implicit function $b$ actually traces the spatial medians as a function of $\varepsilon$, i.e., $t(F_{n, i, j, \varepsilon}) = b(\varepsilon)$ for $\varepsilon \in \mathcal{S}$.

The differentiability of $b$ at zero now entails that we may write $t(F_{n, i, j, \varepsilon}) = b(\varepsilon) = t(F_n) + \varepsilon h_{ij} + o(\varepsilon)$ for some $h_{ij} \in \mathbb{R}^p$ as $ \varepsilon \rightarrow 0$. Plugging the expansion in to the $F_{n, i, j, \varepsilon}$-equivalent of \eqref{eq:gradient_condition}, we get
\begin{align}\label{eq:gradient_condition_perturbed}
\begin{split}
    0 = \sum_{\ell \neq i}^n \frac{y_\ell - \varepsilon h_{ij} + o(\varepsilon)}{\| y_\ell - \varepsilon h_{ij} + o(\varepsilon) \|} + \frac{y_i + \varepsilon (e_j - h_{ij}) + o(\varepsilon)}{\| y_i + \varepsilon (e_j - h_{ij}) + o(\varepsilon) \|},
\end{split}
\end{align}
where $ y_i = x_i - t(F_n) $. The first term in \eqref{eq:gradient_condition_perturbed} splits into two parts, the first of which simplifies as,
\begin{align*}
     \sum_{\ell \neq i}^n \frac{y_\ell}{\| y_\ell - \varepsilon h_{ij} + o(\varepsilon) \|} = \sum_{\ell \neq i}^n \frac{y_\ell}{\| y_\ell \|} + \varepsilon \sum_{\ell \neq i}^n \frac{y_\ell y_\ell' }{\| y_\ell \|^3} h_{ij} + o(\varepsilon),
\end{align*}
whereas, the second part writes,
\begin{align*}
    -  \sum_{\ell \neq i}^n \frac{\varepsilon h_{ij} + o(\varepsilon)}{\| y_\ell - \varepsilon h_{ij} + o(\varepsilon) \|} = - \varepsilon \sum_{\ell \neq i}^n \frac{1}{\| y_\ell \|} h_{ij} + o(\varepsilon).
\end{align*}
Simplifying the second term of \eqref{eq:gradient_condition_perturbed} similarly yields,
\begin{align*}
    & \frac{y_i + \varepsilon (e_j - h_{ij}) + o(\varepsilon)}{\| y_i + \varepsilon (e_j - h_{ij}) + o(\varepsilon) \|} = \frac{y_i}{\| y_i \|} - \varepsilon \frac{y_i y_i'}{\| y_i \|^3} (e_j - h_{ij}) + \varepsilon \frac{1}{\| y_i \|} (e_j - h_{ij}) + o(\varepsilon). 
\end{align*}
Plugging everything now back in to \eqref{eq:gradient_condition_perturbed}, and using \eqref{eq:gradient_condition}, we get, in the notation of the statement of the lemma,
\begin{align*}
    -\varepsilon G h_{ij} + \varepsilon A_i e_j + o(\varepsilon) = 0.
\end{align*}
Division by $\varepsilon$ and letting $\varepsilon \rightarrow 0$ then yields the desired expression for $h_{ij}$.

For $H_{ij}$, resorting to the same notation as in the proof of Lemma \ref{lem:reconstruction_derivative} and writing $y_i^* := x_i^* - t^*$, we first have,
\begin{align*}
    \| y_i^* \|^2 - \| y_i \|^2 = 2 \varepsilon (e_j - h_{ij})'y_i + o(\varepsilon), 
\end{align*}
and,
\begin{align*}
    y_i^* (y_i^*)' - y_i y_i' &= \varepsilon \{ (e_j - h_{ij}) y_i' + y_i(e_j - h_{ij})' \} + o(\varepsilon) =: \varepsilon B_{ij} + o(\varepsilon). 
\end{align*}
For $\ell \neq i$, we also write $y_\ell^* := x_\ell - t^*$ and have
\begin{align*}
     \| y_\ell^* \|^2 - \| y_\ell \|^2 = - 2 \varepsilon h_{ij}' y_\ell + o(\varepsilon),
\end{align*}
and
\begin{align*}
    y_\ell^* (y_\ell^*)' - y_\ell y_\ell' &= -\varepsilon \{  h_{ij} y_\ell' + y_\ell h_{ij}' \} + o(\varepsilon) =: -\varepsilon C_{\ell ij} + o(\varepsilon). 
\end{align*}
Hence, the partition
\begin{align*}
    & S^* - S = \frac{1}{n} \sum_{\ell \neq i}^n \left\{ \frac{y_\ell^* (y_\ell^*)'}{\| y_\ell^* \|^2} - \frac{y_\ell y_\ell'}{\| y_\ell \|^2}  \right\} + \frac{1}{n} \left\{ \frac{y_i^* (y_i^*)'}{\| y_i^* \|^2} - \frac{y_i y_i'}{\| y_i \|^2}  \right\},
\end{align*}
along with the formula for general matrices $A, B$ and scalars $a, b$ with $a \neq 0$ that,
\begin{align*}
    \frac{A + \varepsilon B + o(\varepsilon)}{a + \varepsilon b + o(\varepsilon)} = \frac{A}{a} +  \varepsilon \frac{a B - b A}{a^2} + o(\varepsilon),
\end{align*}
implies that
\begin{align*}
    S^* - S =& -\varepsilon \frac{1}{n} \sum_{\ell = 1}^n \{ \| y_\ell \|^{-2} (  h_{ij} y_\ell' + y_\ell h_{ij}' ) - 2 \| y_\ell \|^{-4} y_\ell h_{ij}' y_\ell y_\ell' \} \\
    +& \varepsilon \frac{1}{n} \{ \| y_i \|^{-2} (  e_{j} y_i' + y_i e_{j}' ) - 2 \| y_i \|^{-4} y_i e_{j}' y_i y_i' \} + o(\varepsilon).
\end{align*}
Simplifying the expression now yields the claim.
\end{proof}

\end{document}